\documentclass{amsart}
\usepackage{amssymb}
\usepackage{graphicx}
\usepackage{amsthm}

\theoremstyle{plain}
\newtheorem{thm}{Theorem}
  
  \newtheorem{lem}[thm]{Lemma}
  \newtheorem{cor}[thm]{Corollary}
  \newtheorem{prop}[thm]{Proposition}
\theoremstyle{remark}
  \newtheorem{rem}[thm]{Remark}
  \newtheorem{exa}[thm]{Example}
  \newtheorem*{notation}{Notation}
\theoremstyle{plain}

\newcommand{\ii}{\mathbf{i}}
\newcommand{\tone}{\psi}
\newcommand{\ttwo}{\varphi}
\begin{document}
\advance\textwidth 3cm
  \title[Polarimetric measurements of nematic liquid crystals]{On one dimensional inverse problems arising from polarimetric measurements of nematic liquid crystals}

\author{Yves Capdeboscq}
\address{Mathematical Institute\\
 24-29 St Giles, OXFORD OX1 3LB\\
 United Kingdom}
\author{Basang Tsering-Xiao}
\address{ Mathematical Department of Tibet University\\
 36 Jiangsu Road, Lhasa, Tibet \\
P.R. China  850000}
\begin{abstract}
We revisit the problem of determining dielectric parameters in layered nematic liquid crystals from polarimetric measurements originally 
introduced by Lionheart \& Newton. 
After a detailed analysis of the model, of the scales involved, and of natural obstacles to the reconstruction of more than one dielectric 
parameters,  we produce two simple one-dimensional inverse problems which can be studied without any expertise in liquid crystals. 
We then confirm that very little can be recovered about the internal configuration of smooth dielectric parameters from these measurements, 
and give a uniqueness result for one of the two problem, when the unknown parameter satisfies a monotonicity property. 
In that case, the available data can be expressed in terms of Laplace and Hankel transforms.
\end{abstract}

\maketitle

\section{Introduction}

\begin{figure}\label{fig:setup}
\centerline{\includegraphics[width=0.7\columnwidth]{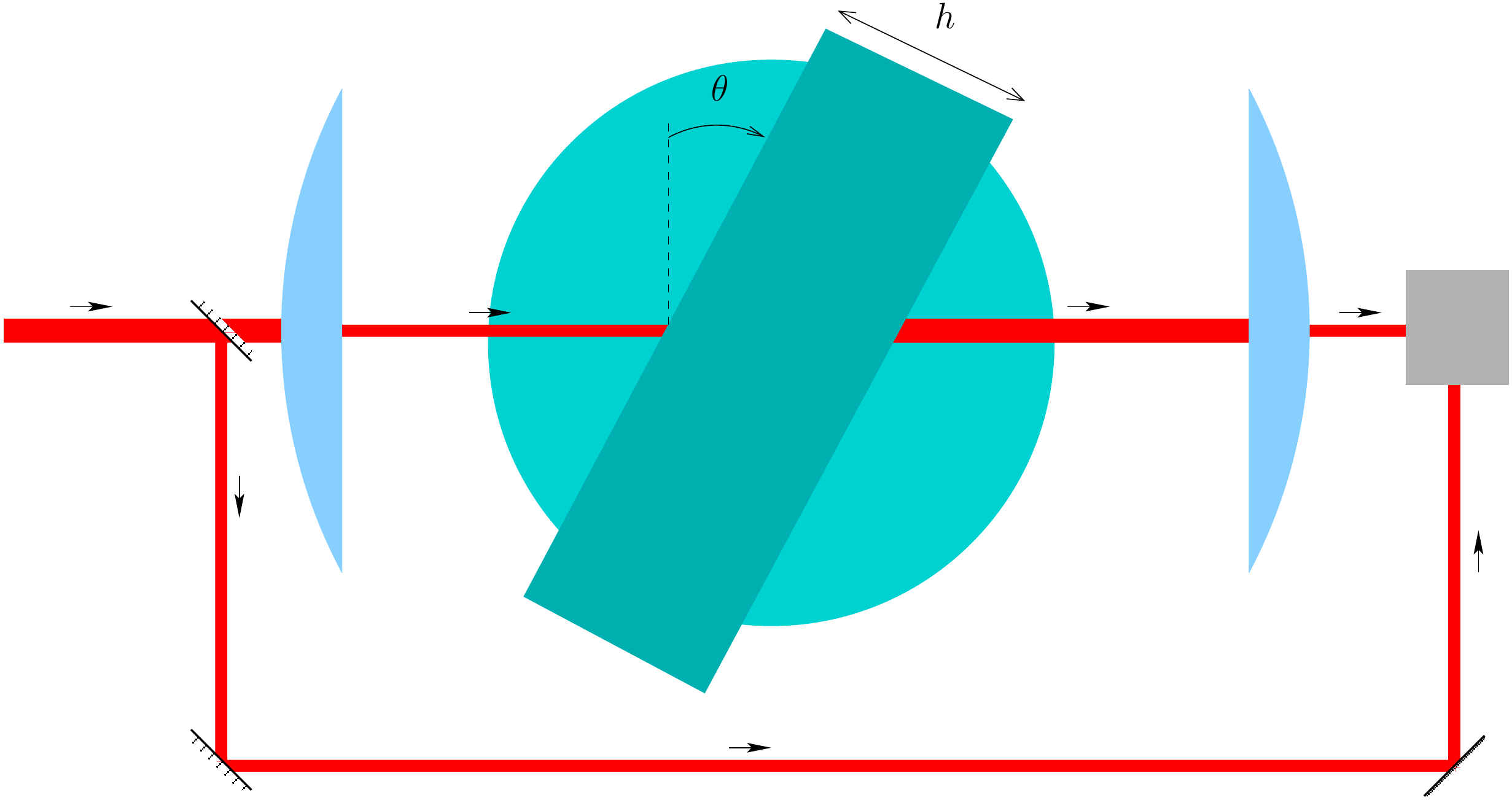}}
\caption{Sketch of the polarimetric measurement experiment.
A polarized laser beam is split in two, the main part is focused, sent through a thin slab containing
the liquid crystal. The slab is placed in an isotropic, homogeneous medium whose index is close to the 
average index of the slab. The transmitted beam is refocused, and compared to the initial beam by an 
analyser. The planar cross-section sketched above is referred to as the ($x_1$--$x_3$) plane in the text.}
\end{figure}

Consider the experiment sketched in Figure~\ref{fig:setup}.
A polarized focused laser beam is split in two beams.
One beam propagates in an homogeneous medium (in light blue),
and travels through a slab containing the liquid crystal (LC) cell.
The transmitted beam exits on the opposite side of the slab, and enters an analyser,
where it is compared to the other part of the split beam, which did not pass through
the slab. The analyser delivers four real numbers, called the Stokes parameters. The LC cell is
placed on a cylindrical mount, to allow variations of the incident angle of the laser in the slab.
The Stokes parameters, appropriately normalized, are collected
for a large range of incident  angles. The purpose of this article is to study the dependence
of the data on the dielectric tensor of the liquid crystal, and discuss its
possible reconstruction.

The liquid crystal contained in the slab has a layered structure: the dielectric permittivity is
constant and anisotropic in each layer. The magnetic permeability is a constant number. The liquid
crystal is coated with thin films (polyimide, indium tin oxide...) but these thin films will not be
taken into account in our study. This experiment was performed at Hewlett-Packard Laboratories,
Bristol and was analysed by Lionheart \& Newton~\cite{LIONHEART-NEWTON-07} using a singular 
value decomposition approach.

To model this problem we adopt the approximation introduced in Berreman~\cite{BERREMAN-72}, following
Lionheart \& Newton~\cite{LIONHEART-NEWTON-07}. The incident polarized laser beam in the isotropic 
medium in front of the liquid crystal is modelled by monochromatic plane waves incident obliquely in 
the ($x_1$--$x_3$) plane (the plane in which the experiment is sketched), with direction
$\overrightarrow{s}=\left(\sin(\theta),0,\cos(\theta)\right)$.
The corresponding electric field has two components,
\[
\mathcal{E}^{i}_{1}=E_{1}\exp\left(\ii\omega
\left(\frac{n_{0}}{c}\overrightarrow{s}\cdot\overrightarrow{x}\right)\right)
\]
which is parallel to $\overrightarrow{s}^{1}_{\perp}=(\cos(\theta),0,-\sin(\theta))$, and
\[
\mathcal{E}^{i}_{2}=E_{2}\exp\left(\ii\omega
\left(\frac{n_{0}}{c}\overrightarrow{s}\cdot\overrightarrow{x}\right)\right)
\]
along $\overrightarrow{s}^{2}_{\perp}=(0,1,0)$. To take into account possible phase differences,
the amplitudes of $E_{1}$ and $E_{2}$ are complex numbers. When the incident field is polarized linearly,
$E_1$ and $E_2$ are of the form $A \cos(\phi)$ and $A \sin(\phi)$, with  $A\in \mathbb{C}$ and $\phi\in[0,2\pi)$.

When the incident field has a circular polarization, $E_2= \pm i E_1$. The central frequency of
the laser is denoted by $\omega$, $c$ is the speed of light, and $n_0$ is the index of the isotropic
medium around the slab on the wave path.
We assume that within the slab, the propagating electric and magnetic fields $\mathcal{E}$
and $\mathcal{H}$ stay alike plane waves along the $x_1$ direction, and do not depend on $x_2$.
Thus $\mathcal{E}$ and $\mathcal{H}$ are modelled by the following ansatz
\begin{equation}
 \mathcal{E}=E(x_{3})\exp \lbrace
\ii\omega(\frac{n_0}{c}s_{1}x_{1})\rbrace,\quad
\mathcal{H}=H(x_{3})\exp \lbrace
\ii\omega(\frac{n_0}{c}s_{1}x_{1})\rbrace, \label{eq:ansatz}
\end{equation}
and satisfy the time-harmonic Maxwell's equations
\begin{equation}
\nabla\land\mathcal{E}  =\ii\omega \mu_0 \mathcal{H},\quad
\nabla\land\mathcal{H}  =-\ii\omega \varepsilon_0 \varepsilon (x_3)\mathcal{E}.\label{eq:max2}
\end{equation}
Where $\varepsilon$ is the relative dielectric tensor, with constant entries outside the slab and
varying as a function of $x_3$ within the LC cell, and $\mu_0$ and $\varepsilon_0$ are
universal constants. See Lionheart \& Newton~\cite{LIONHEART-NEWTON-07}, Tsering-Xiao~\cite{BASANG-THESIS} and references therein
for more details and discussions on this model, and Lavrentovich~\cite{LAVRENTOVICH-12} for further insight on optical 
measurements of liquid crystals.
Substituting \eqref{eq:ansatz} into \eqref{eq:max2}, Maxwell's system becomes
a $4\times4$ system of ordinary differential equations in $x_3$ where
the unknown is the vector $X=(\mathcal{E}_1, \mathcal{H}_2,\mathcal{E}_2,-\mathcal{H}_1)^{T}$
which satisfies
\begin{equation}
\frac{dX}{dx_3}=\ii\frac{\omega}{c}M(x_3)X(x_3).
\label{eq:originBerreman}
\end{equation}
The  Berreman vector $X$ is the tangential part of the electromagnetic fields and represents the optical field.
The Berreman matrix $M$ depends on the dielectric tensor $\varepsilon$.
In the outer isotropic medium,  the Berreman matrix has a block diagonal structure, corresponding to the
two possible polarizations of the electric field, along $\overrightarrow{s}^{1}_{\perp}$ and $\overrightarrow{s}^{2}_{\perp}$.
 The Berreman matrix has
four eigenvectors $(V_i)_{i=1,..,4}$.  The vectors $V_1$ and $V_3$ correspond to the incoming electric and magnetic fields,
whereas the vectors $V_2$ and $V_4$ correspond  to the  outgoing electric and magnetic fields, see Section~\ref{sec:form0}.
The direct problem is formulated in Lionheart \& Newton~\cite{LIONHEART-NEWTON-07} as a transmission problem.
The incident  field and the transmitted field are  represented by the input
vector $X^i = \mathcal{E}^{i}_{1} V_1 + \mathcal{E}^{i}_{2} V_3$ and the output vector
$X^t=\mathcal{E}^{t}_{1} V_1 + \mathcal{E}^{t}_{2} V_3$.
The reflected field $X^r$ is a combination of the outgoing eigenvectors $V_2$ and $V_4$.
If $h$ is the thickness of the slab, and
$P(h)$ is the transmission matrix coming
from \eqref{eq:originBerreman}, the transmission problem takes the form
\begin{equation}\label{eq:trans-LN}
P(h)(X_i +X_r)=X_t.
\end{equation}
The Stokes parameters (the measured data) are
\begin{equation}\label{eq:Stokes-LN}
S (E^{i}_1,E^{i}_2,\theta,\varepsilon)
=\left(\left|\mathcal{E}^{t}_{1}\right|^2+\left|\mathcal{E}^{t}_{2}\right|^2,
\left|\mathcal{E}^{t}_{1}\right|^2-\left|\mathcal{E}^{t}_{2}\right|^2,2 \mathcal{E}^{t}_{1}\overline{\mathcal{E}^{t}_{2}}\right).
\end{equation}
The questions investigated numerically in Lionheart \& Newton~\cite{LIONHEART-NEWTON-07} are
\begin{enumerate}
\item Can the dielectric tensor through the LC cell be deduced from the data?
\item How is the solution affected by the range of incident angles and input polarisations used?
\item Given the limited accuracy of the polarimeter  how much information can be deduced from the data?
\end{enumerate}
The goal of this paper is to address these questions analytically. In section~\ref{sec:formulation}
we reformulate the problem in an equivalent, non-dimensional form, taking into account the scale of the various parameters
of the problem.  In section~\ref{sec:form0} we detail the derivation of the Berreman model \eqref{eq:originBerreman},
 given in a non-dimensional form by \eqref{eq:trans-2}.

We consider two types of LC cells, orthorhombic and nematic uniaxial. In section~\ref{sec:orthlc} we investigate the case of an
orthorhombic medium with principle axes aligned with the coordinate axes, that is, we assume that $\varepsilon$ is a diagonal
matrix-valued function.  This first model was also discussed  in Lionheart \& Newton~\cite{LIONHEART-NEWTON-07}, with a different approach. 
Orthorhombic crystals are one of the typical anisotropic materials. Within the class of liquid crystals, orthorhombic symmetry was 
considered as a convenient theoretical possibility in early developments, see e.g. Freiser~\cite{FREISER-70}. Since then, 
it has been observed for specific bi-axial liquid crystals by Hegmann et al.~\cite{HEGMANN-01,KASNACHEEV-07}, but it is deemed to be 
very rare see e.g. Karahaliou et al.~\cite{KARAHALIOU-VANAKARAS-PHOTINOS-09}. In section~\ref{sec:uniLC}, we turn to the case of a nematic 
uniaxial LC cell. This is a very common  model for LC cells. 
In that case, $\varepsilon$ can be expressed in terms of a director profile, that is
\begin{equation}
\varepsilon=\varepsilon_{\perp}I_d+(\varepsilon_{\parallel}-\varepsilon_{\perp})\textbf{n}\otimes
\textbf{n}\, , \label{eq:uniaxialDielectric}
\end{equation}
where, for each $x_3$, $\textbf{n}(x_3)\in\mathbb{S}^2$ is a unit vector, the so-called  director vector of the LC cell.
The two constant eigenvalues of $\varepsilon$ are usually written $\varepsilon_{\perp}$
and $\varepsilon_{\parallel}$. The numbers $\sqrt{\varepsilon_{\perp}}$ and $\sqrt{\varepsilon_{\parallel}}$
are known as the refractive indices of ordinary and extraordinary
waves. The subscripts $\parallel$ and $\perp$ refer to field
directions respectively parallel and perpendicular to the director
vector $\textbf{n}$. Formula \eqref{eq:uniaxialDielectric}  implies that the electric field energy propagates
separately along the ordinary and extraordinary direction. The derivation of this model and further background
can be found in Virga~\cite{VIRGA-94}, de Gennes \& Prost~\cite{DEGENNES-PROST-93} and Chandrasekhar~\cite{CHANDRASEKHAR-92}.

For both models, we highlight intrinsic obstacles to the determination of the dielectric parameters: non-uniqueness is endemic,
 regardless of practical limitations such as the range and precision of the measurements involved. In section~\ref{sec:scales},
we express the scales of the various coefficients involved in the experiment described in Lionheart \& Newton~\cite{LIONHEART-NEWTON-07} in terms of a
non-dimensional parameter.

We conclude this section by producing three one dimensional inverse problems coming from both the orthorhombic LC cell and nematic uniaxial LC
cell model, Problem~A given by \eqref{eq:problemeA}, Problem~B given by \eqref{eq:defDB} and Problem~C given by \eqref{eq:defDC}.
Problem~A and Problem~B concern the reconstruction of one coefficient from the amplitude of the transmission data.
This coefficient would be  $\varepsilon_{22}$ for the orthorhombic LC cell example, and one component of director profile
$\textbf{n}$ {\it{a priori}} assumed to lie  in the ($x_1$--$x_3$) incident plane.
Problem~C presents the independent phase data available from the nematic uniaxial LC cell model, which is just one
numerical value.

Section~\ref{sec:ProblemA} is devoted to Problem~A. Under a smoothness assumption of the coefficient to be reconstructed,
we describe what can be extracted from the data corresponding to moderate angles of incidence. We then explain how 
Problem~B connects to Problem~A, as a first order approximation, when the full range of angles of incidence are considered.

Section~\ref{sec:ProblemB} is devoted to Problem~B. For this model problem, we can describe the class of equivalent parameters,
and give a uniqueness result with a monotonicity assumption. We show that this problem can be reformulated in terms of
the classical Laplace and Hankel transforms.

Section~\ref{sec:proofss} contains the proofs of various technical intermediate results given in the previous sections.
We summarize our findings and discuss possible extensions in section~\ref{sec:concl}. We use some facts concerning systems of
ordinary differential equations in this paper, they are given in appendix for the reader's convenience.

\section{\label{sec:formulation} Modelling and scaling assumptions}

\subsection{\label{sec:form0}An alternative equivalent formulation}

In this section, we prove the following proposition.

\begin{prop}\label{pro:equiv-form}
The transmission problem \eqref{eq:originBerreman}--\eqref{eq:trans-LN} can be written in a non-dimensional form as follows. Given $I_1$ and $I_2$ in $\mathbb{C}$,
find $T_1,T_2,R_1$ and $R_2$ in $\mathbb{C}$  such that there exists a continuous solution of
\begin{eqnarray}
\frac{d Z}{d t} &=&\ii\frac{1}{\eta}B(t)Z(t), \quad t \in [0 , 1] ,\nonumber \\
Z(0)&=&\left[I_1-R_1,I_1+R_1,I_2-R_2,I_2+R_2\right]^{T}, \label{eq:trans-2} \\
Z(1)&=&\left[T_1,T_1,T_2,T_2\right]^{T},\nonumber
\end{eqnarray}
where $B$ is a matrix-valued piecewise continuous function of $t\in [0,1]$ given by
\begin{equation}
B=\frac{1}{\tilde{\varepsilon}_{33}}\times  \left[\begin{array}{cccc}
- \sin(\theta) \tilde{\varepsilon}_{13}  & \frac{\displaystyle \lambda^2 + \tilde{\varepsilon}_{33}-1}{\lambda} &0&
-\sin(\theta)\lambda^{-1}\tilde{\varepsilon}_{{23}}  \\
 \lambda \left(\tilde{\varepsilon}_{{11}}\tilde{\varepsilon}_{{33}}-{\tilde{\varepsilon}_{{13}}}^{2}\right)&
- \sin(\theta) \tilde{\varepsilon}_{{13}}
& 0 &  \tilde{\varepsilon}_{{12}}\tilde{\varepsilon}_{{33}}-\tilde{\varepsilon}_{{13}}\tilde{\varepsilon}_{{23}} \\
\tilde{\varepsilon}_{{12}}\tilde{\varepsilon}_{{33}}-\tilde{\varepsilon}_{{13}}\tilde{\varepsilon}_{{23}}
& - \sin(\theta){\lambda}^{-1} \tilde{\varepsilon}_{{23}}  &0&
\frac{\displaystyle(\lambda^2 +\tilde{\varepsilon}_{{22}}-1)\tilde{\varepsilon}_{{33}}-{\tilde{\varepsilon}_{{23}}}^{2}}{\lambda}
\\ 0&0&
\lambda \tilde{\varepsilon}_{33}&0\end{array}\right], \label{eq:def-B}
\end{equation}
with
\begin{equation}\label{eq:eta-theta}
t= \frac{x_3}{h},\quad \tilde\varepsilon = \frac{1}{(n_0)^2}\varepsilon,\quad  \frac{\omega n_0 h}{c} = \eta^{-1},\mbox{ and } \cos(\theta)= \lambda.
\end{equation}
The length $h$ is the thickness of the LC slab, $\theta$ is the angle of incidence, and $n_0$ is the index of the homogeneous isotropic medium surrounding the slab.

\medskip{}

The Stokes parameter data given by \eqref{eq:Stokes-LN} is equivalent to
\begin{equation}\label{eq:data-stokes}
\tilde{\mathcal{S}}(I_1,I_2,\theta)=(T_1\exp\left(\ii f(I_1,I_2,\theta)\right),T_2\exp\left(\ii f(I_1,I_2,\theta)\right)),
\end{equation}
where the real function $f$ is unknown, and may depend on $I_1,I_2$ and $\theta$.
\end{prop}
Problem~\eqref{eq:trans-2} has only partial conditions on the solution
at the start point and  partial conditions at the end point. It is
therefore not a standard initial eigenvalue problem for system of ordinary differential equations
\eqref{eq:trans-1}. It is nevertheless well-posed, as the following proposition shows, proved in \ref{sec:appendixA} for completeness.
\begin{prop}\label{pro:well-posed-ini}
Given $I_1$ and $I_2$, there exist
a unique pair of reflection parameters $R_1$ and $R_2$, a unique pair of transmission parameters $T_1$ and $T_2$, and
a unique $Z$ solution of \eqref{eq:trans-2}.
\end{prop}

\begin{proof}[Proof of proposition~\ref{pro:equiv-form}]
We rescale the space variable $x_3$ by $t=x_3/h$. Problem~\ref{eq:originBerreman} becomes
\begin{equation}
\frac{dX}{dt}= \frac{\ii}{\eta}\tilde{M}(t)X(t),
\label{eq:Berreman2}
\end{equation}
and  $\tilde M$ is given by
\begin{equation}
 \tilde{M}=\left[\begin{array}{cccc}
-\frac{\tilde\varepsilon_{{13}}\sin(\theta) }{\tilde\varepsilon_{{33}}}
 & \frac{\tilde\varepsilon_{{33}}-\sin(\theta)^{2}}{n_0 \varepsilon_{0}c\tilde\varepsilon_{{33}}}
 & -\frac{\tilde\varepsilon_{{23}}\sin(\theta)}{\tilde\varepsilon_{{33}}}
 & 0
\\
n_0\frac{\varepsilon_{0}c\left(\tilde\varepsilon_{{11}}\tilde\varepsilon_{{33}}-{\tilde\varepsilon_{{13}}}^{2}\right)}{\tilde\varepsilon_{{33}}}
& -\frac{\tilde\varepsilon_{{13}} \sin(\theta)}{\tilde\varepsilon_{{33}}}
& n_0\frac{\varepsilon_{0}c\left(\tilde\varepsilon_{{12}}\tilde\varepsilon_{{33}}-\tilde\varepsilon_{{13}}\tilde\varepsilon_{{23}}\right)}{\tilde\varepsilon_{{33}}}
& 0
\\
0 & 0 & 0 & \frac{1}{n_0 \varepsilon_{0}c}\\
n_0 \frac{\varepsilon_{0}c\left(\tilde\varepsilon_{{12}}\tilde\varepsilon_{{33}}-\tilde\varepsilon_{{13}}\tilde\varepsilon_{{23}}\right)}{\tilde\varepsilon_{{33}}}
& -
\frac{\tilde\varepsilon_{{23}}\sin(\theta)}{\tilde\varepsilon_{{33}}} &
n_0
\frac{\varepsilon_{0}c\left(\tilde\varepsilon_{{22}}\tilde\varepsilon_{{33}}-{\tilde\varepsilon_{{23}}}^{2}-\sin(\theta)^{2}\tilde\varepsilon_{{33}}\right)}{\tilde\varepsilon_{{33}}}
& 0
\end{array}\right]. \label{eq:tildeM}
\end{equation}
In the outer isotropic medium, the matrix $\tilde M$ simplifies to
$$
M_{1}=\left[\begin{array}{cccc}
0 & \frac{1}{n_0\varepsilon_{0}c}\lambda & 0 & 0\\
n_{0} \varepsilon_{0}c & 0 & 0 & 0\\
0 & 0 & 0 & \frac{1}{n_0 \varepsilon_{0}c}\\
0 & 0 & n_{0}\varepsilon_{0}c\lambda &
0\end{array}\right],\label{eq:outerM}
$$
which has a block diagonal structure, corresponding to the
independence of two possible polarisations of the electric field.
The matrix $M_1$ has
four eigenvectors, two corresponding to the incoming electric and
magnetic fields, and two corresponding to the outgoing electric and
magnetic fields. Ordering them in agreement with the block
structure, these are
\[
V=\left[\begin{array}{cccc}
{\displaystyle \lambda} & {\displaystyle -\lambda} & {\displaystyle 0} & {\displaystyle 0}\\
{\displaystyle \varepsilon_{0}cn_{0}} & {\displaystyle \varepsilon_{0}cn_{0}} & {\displaystyle 0} & {\displaystyle 0}\\
{\displaystyle 0} & {\displaystyle 0} & {\displaystyle 1} & 1\\
{\displaystyle 0} & {\displaystyle 0} & {\displaystyle
\varepsilon_{0}cn_{0}\lambda} & {\displaystyle
-\varepsilon_{0}cn_{0}\lambda}\end{array}\right]
\]
 corresponding to the eigenvalues $\lambda,-\lambda,\lambda,-\lambda$. We note that
\[
X^{i}=
\mathcal{E}^{I}_{1} \,V
\left[
1,
0,
0,
0\right]^{T}
+
\mathcal{E}^{I}_{2}\, V
\left[
0,
0,
1,
0\right]^{T}.
\]
The transmission problem takes a simpler form if the matrix $\tilde M$ is written in the eigenbasis of $M_1$, that is,
\begin{equation}
\hat{M}=V^{-1}\tilde{M}V.\label{eq:defmtilde}
\end{equation}
The transmission problem \eqref{eq:originBerreman}--\eqref{eq:trans-LN} then becomes
\begin{quote}
{\emph{Given an initial condition
$I=\left[I_{1},0,I_{2},0\right]^{T}$,
find $T=\left[T_{1},0,T_{2},0\right]^{T}$
and
$R=\left[0,R_{1},0,R_{2}\right]^{T}$ such that the solution of
\begin{equation}\label{eq:trans-1}
\frac{dY}{dt}=\frac{\ii}{\eta}\hat{M}(t)Y(t), \quad t \in [0 , 1]
, \end{equation}
satisfies\[ {}Y(0)=I+R, \quad Y(1)=T.\]}}
\end{quote}
{The stokes parameter are then given by}
$$
S(I_1,I_2,\theta)
=\left(\left|{T}_{1}\right|^2+\left|{T}_{2}\right|^2,
\left|{T}_{1}\right|^2-\left|{T}_{2}\right|^2,2 {T}_{1}\overline{{T}_{2}}\right).
$$
A computation shows that the available data is equivalent to the knowledge of the vector $\mathcal{S}$ given by \eqref{eq:data-stokes}.
Finally, the structure of the matrix $\hat{M}$, is simpler after a block $-\pi/4$  a rotation.
We define
\[
B=R_{\frac{\pi}{4}}( \hat{M} )R_{-\frac{\pi}{4}},
\]
where \[{}R_{\frac{\pi}{4}}=\frac{1}{\sqrt2}
\left[\begin{array}{cccc}
{1}&{-1}&{0}&{0}\\
{1}&{1}&{0}&{0}\\
{0}&{0}&{1}&{-1}\\
{0}&{0}&{1}&{1}
\end{array}\right],\]
and the entries of $B$ are given by \eqref{eq:def-B}. After these simplifications, problem \eqref{eq:trans-1} transforms into problem \eqref{eq:trans-2}.
\end{proof}

\subsection{Orthorhombic LC cell model}\label{sec:orthlc}
The following proposition summarizes our findings concerning the orthorhombic model.
\begin{prop}\label{pro:nonu-orth}
If $\varepsilon$ is a diagonal matrix with three independent piecewise continuous functions $(\varepsilon_{ii})_{i=1,2,3}$
as diagonal entries, not all three can be determined from the Stokes parameters: $\varepsilon_{11}$ and $\varepsilon_{33}$
cannot be determined independently from each other. On the other hand, the transmission coefficient $T_1$
is independent of $\varepsilon_{22}$, whereas $T_2$ is uniquely determined by the system
\begin{eqnarray}
\eta\lambda \frac{d z_{3}}{d t} & =& \ii (\lambda^2 +\tilde{\varepsilon}_{22}-1) z_{4} \mbox{ in }  [0,1], \nonumber\\
 \eta \frac{d z_{4}}{d t} & =& \ii \lambda  z_{3}  \mbox{ in } [0,1], \label{eq:probleme2-0}\\
z_{3}(0) & =& 1-R_{2},\quad z_{4}(0)  =1+R_{2},\quad z_{3}(1) =z_{4}(1)=T_{2}. \nonumber
\end{eqnarray}
If $\varepsilon_{11}$ and $\varepsilon_{33}$  are unknown, the available data to determine $\varepsilon_{22}$ is
$|T_2|$, where $|\cdot|$ denotes the complex modulus.
If the phase of $T_1$ is known, and this is the case when both $\varepsilon_{11}$ and $\varepsilon_{33}$
are known, the available data is $T_2$.
\end{prop}
\begin{proof}
If we assume that $\varepsilon$ is a diagonal matrix-valued function with $\varepsilon_{ii}>C>0$ for $i=1,2,3$
and $C$ is a positive constant, the Berreman matrix $B$ becomes block diagonal:
\[
B = \left[\begin{array}{cccc}
0 & {\displaystyle \frac{\lambda^2+\tilde{\varepsilon}_{33}-1}{\tilde{\varepsilon}_{33}\lambda} }& 0 & 0\\
\lambda \tilde{\varepsilon}_{11} & 0 & 0 & 0\\
0 & 0 & 0 & {\displaystyle \frac{\lambda^2 + \tilde{\varepsilon}_{22}-1}{\lambda}}\\
0 & 0 & \lambda & 0\end{array}\right].
\]
By inspection we note that the unknowns $(R_1,T_1)$ only depend on $I_1$, and the unknowns $(R_2,T_2)$ only depend on $I_2$: the two transmission
modes are decoupled. By linearity, we can thus set $I_1=I_2=1$. The transmission problem relating $R_2$ and $T_2$ becomes \eqref{eq:probleme2-0}.

To write the transmission problem relating $R_1$ and $T_1$ in a similar form, we write
$\Lambda_1=\int_0^{1} \tilde{\varepsilon}_{11}(s)\,ds$, and introduce the change of
variable  $\Lambda_1 \nu(x)= \int_0^{x}  \tilde{\varepsilon}_{11}(s)\,ds$, and write
the inverse change of variable $\mu=\nu^{-1}$.
The first block of \eqref{eq:trans-2} then becomes
\begin{eqnarray}
\frac{\eta\lambda}{\Lambda_1} \frac{d z_{1}}{d \nu} & =& \ii
 \frac{\lambda^2 +\tilde{\varepsilon}_{33}(\mu(\nu))-1}{\tilde{\varepsilon}_{11}(\mu(\nu))\tilde{\varepsilon}_{33}(\mu(\nu))} z_{2}
 \mbox{ in } [0,1], \nonumber\\
 \frac{\eta}{\Lambda_1} \frac{d z_{2}}{d \nu} & =& \ii \lambda z_{1}  \mbox{ in } [0,1], \label{eq:probleme1-0}\\
z_{1}(0) & =& 1-R_{1},\quad z_{2}(0)  =1+R_{1},\quad z_{1}(1) =z_{2}(1)=T_{1}. \nonumber
\end{eqnarray}
It is clear that problem~\eqref{eq:probleme1-0} contains too many unknown functions and parameters for each of them to be
uniquely determined by the map $\lambda \to T_1$.
Furthermore, formula \eqref{eq:data-stokes} given in proposition~\ref{pro:equiv-form} shows that for each $\lambda$
the pair $(T_1,T_2)$ is only known up to an arbitrary phase  shift. Since $T_1$ and $T_2$ depend on different unknown functions,
in general only $|T_1|$ and $|T_2|$ are available.
\end{proof}
\begin{prop}\label{pro:nonu-orth-2}
The transmission coefficient $T_2$ is unchanged if $t\to\varepsilon_{22}(t)$ is replaced by $t\to\varepsilon_{22}(1-t)$.
\end{prop}
We check this classical reversibility property in \ref{Annex:B}.

\subsection{Uniaxial nematic LC cell model}\label{sec:uniLC}
Let us now consider the nematic uniaxial model \eqref{eq:uniaxialDielectric}. The index of the surrounding homogeneous medium is given by
\begin{equation}\label{eq:defno}
n_0^2 = \sqrt{\varepsilon_{\perp}\varepsilon_{\parallel}}.
\end{equation}
We parametrize the director vector by
\begin{eqnarray}
\textbf{n}=(\cos(\tone)\cos(\ttwo),\cos(\tone)\sin(\ttwo),\sin(\tone)), \nonumber\\
  (\tone,\ttwo) \in (-\frac{\pi}{2},\frac{\pi}{2}]\times[0,\pi).\label{eq:def-n}
\end{eqnarray}
The tilt angle $\tone$ and the azimuthal angle $\ttwo$ parametrize a half-sphere only, since $\varepsilon$ is unchanged
if $\textbf{n}$ is changed into $-\textbf{n}$. Uniaxial configurations are {\it a priori} less complex than the
orthorhombic ones, since two functions instead of three are to be determined. We simplify the problem even further,
and impose that $\ttwo=0$: the director vector $\textbf{n}$ lies in the incident plane.
This is still not sufficient for uniqueness, as shown by the following proposition.
\begin{prop}\label{pro:nonu-nema}
If $\varepsilon$ is the dielectric permittivity matrix of a  nematic uniaxial LC cell modelled by \eqref{eq:uniaxialDielectric},
the Stokes parameter data are not sufficient to determine uniquely the tilt angle $\tone$ and the azimuthal angle $\ttwo$ of the director
vector defined in \eqref{eq:def-n}. When the director vector stays in the incident plane, that is $\ttwo\equiv 0$, the available data are
$$
T(\lambda)F(\lambda)\mbox{ and  } {T(\lambda) }/{F(\lambda)},
$$
where
\begin{equation}\label{eq:K1}
F(\lambda) =  \exp\left( \ii \eta^{-1}\delta \sqrt{1-\lambda^2} \int_0^{1} \frac{\displaystyle \sin(2 \tone (s))}
{\displaystyle 1+\delta  \cos(2 \tone (s)) }ds \right).
\end{equation}
The parameters with $\lambda$, and $\eta$ are defined in \eqref{eq:eta-theta}, whereas $\delta$ is given by
\begin{equation}\label{eq:delta-nbar}
\delta = \frac{\varepsilon_{\perp} -\varepsilon_{\parallel}}{\varepsilon_{\perp}+\varepsilon_{\parallel}}.
\end{equation}
The map  $\lambda\to T(\lambda)$ is defined by 
\begin{eqnarray}
\tilde{\eta}\lambda \frac{d u_{1}}{d t} & =& \ii
\left(\lambda^2+ \frac{\delta}{\sqrt{1-\delta^2}} \left(\cos(2\psi(\mu)) + \frac{\delta}{1+\sqrt{1-\delta^2}}\right)\right) u_2  \mbox{ in } [0,1], \nonumber\\
\tilde{\eta} \frac{d u_{2}}{d t} & =& \ii  \lambda u_{1}
\mbox{ in } [0,1], \label{eq:probleme-tilt-N}\\
u_{1}(0) & =& 1-R(\lambda),\quad u_{2}(0)  =1+R(\lambda), \quad u_{1}(1) = u_{2}(1)=T(\lambda).\nonumber
\end{eqnarray}
where $\mu(\tau(t))=t$ for all $t\in[0,1]$, with
$$
\tau(t) = \frac{1}{\mathcal{N}} \int_0^t \frac{1}{1+ \delta\cos(2 \tone(s) ) } ds, \quad \mathcal{N} = \int_0^1 \frac{1}{1+ \delta\cos(2 \tone(s)) } ds,\quad
\tilde{\eta}=\frac{\eta}{\sqrt{1-\delta^2} \mathcal{N}}.
$$
The  map $t\to\tone(t)$ is not uniquely determined by the available data.
\end{prop}
\begin{proof}
When  $\ttwo=0$, the Berreman matrix $B$ is block diagonal,
with
\[
(B)_{1..2,1..2} = \left[\begin{array}{cc}
 \displaystyle \frac{\delta \sin(\theta)\sin(2\tone )}{\displaystyle \delta \cos(2\tone ) + 1}
 &\displaystyle \frac{1}{\lambda}\left(1+ \frac{\sqrt{1-\delta^2}(\lambda^2-1)}{\displaystyle \delta \cos(2\tone ) + 1}\right) \\
\displaystyle  \frac{ \sqrt{1-\delta^2} \lambda}{\displaystyle \delta \cos(2\tone ) + 1}
&
\displaystyle \frac{\delta \sin(\theta)\sin(2\tone )}{\displaystyle \delta \cos(2\tone ) + 1}
\end{array}\right],\label{berre-yves1}
\]
and
\[
(B)_{3..4,3..4}=\left[\begin{array}{cc}
0 & {\displaystyle \frac{n_0^2(\lambda^2-1) + \varepsilon_{\perp}}{n_0^2 \lambda}}\\
 \lambda & 0\end{array}\right].\label{berre-yves2}
\]
The second transmission parameter, $T_2$, can be computed explicitly independently of the tilt angle $\tone $.
Thus, using formula~\eqref{eq:data-stokes} we see that the available data to determine $\tone $ is $T_1$, and not just its modulus.
Performing the change of unknown
$$
(u_1(t),u_2(t)) = \exp\left(-\ii\frac{\sin(\theta)}{\eta} \int_0^{t} \frac{\displaystyle \delta \sin(2\tone (s))}
{\displaystyle  \delta \cos(2\tone (s))+1 }ds\right) (z_1(t),z_2(t)),
$$
we see that $T_1$ is given by  $T(\lambda) F(\lambda)$ or $T(\lambda)/F(\lambda)$  with
 $T(\lambda)$ determined by
\begin{eqnarray}
\eta\lambda \frac{d u_{1}}{d t} & =& \ii
\left(1 + (\lambda^2-1)\frac{\displaystyle \sqrt{1-\delta^2}}{1+ \delta\displaystyle\cos(2 \tone  )}
\right) u_2  \mbox{ in } [0,1], \nonumber\\
\eta \frac{d u_{2}}{d t} & =& \ii  \lambda\frac{\displaystyle\sqrt{1-\delta^2}}{1 +\delta \displaystyle  \cos(2 \tone )} u_{1}
\mbox{ in } [0,1], \label{eq:problemetilt-0}\\
u_{1}(0) & =& 1-R(\lambda),\quad u_{2}(0)  =1+R(\lambda), \quad u_{1}(1) = u_{2}(1)=T(\lambda).\nonumber
\end{eqnarray}
Since both positive and negative incident angles are measured, but  $T(\lambda)$ only depends on $\lambda=\cos(\theta)$,
the data $T_1$ is both $T(\lambda) F(\lambda)$ and $T(\lambda)/F(\lambda)$.

We then change $t$ to $\tau$ in \eqref{eq:problemetilt-0}  and obtain \eqref{eq:probleme-tilt-N} (where we named $t$ instead $\tau$ the dummy variable).

The coefficient $T(\lambda)$  depends on $\tone $, through $\cos(2\tone )$ :  positive and  negative tilt angles are not distinguishable.
The phase difference between $T(\lambda)$ and $T(\lambda)F(\lambda)$ or $T(\lambda)/F(\lambda)$ depends on $\tone $ via the number 
$$
K =  \int_0^{1} \frac{\displaystyle \sin(2 \tone (s))}
{\displaystyle 1+\delta  \cos(2 \tone (s)) }ds.
$$
In Figure~\ref{fig:equiv} we show the graph of two tilt angle functions which would have identical transmission parameter
$T(\lambda)$ for all $\lambda$. To construct these examples, we chose to repeat a given compactly supported pattern five times.
In two instances out of five, the pattern is flipped with respect to the $t$ axis. There are ten possibilities: we chose
two different ones arbitrarily. The parameter $T(\lambda)$ is equal for both graphs, as it is invariant
under arbitrary changes of the tilt angle sign.
The constant $K$ sums the contribution of the five patterns, three with a plus sign and two with a minus sign. The way these patterns
are ordered does not change this integral.
This non-uniqueness comes in addition to the one already highlighted in proposition~\ref{pro:nonu-orth-2},
namely that $T(\lambda)$ is unchanged if $t\to\tone(t)$ is replaced by $t\to\tone(1-t)$, see \ref{Annex:B}.
\end{proof}
\begin{figure}
\includegraphics[width=0.9\columnwidth]{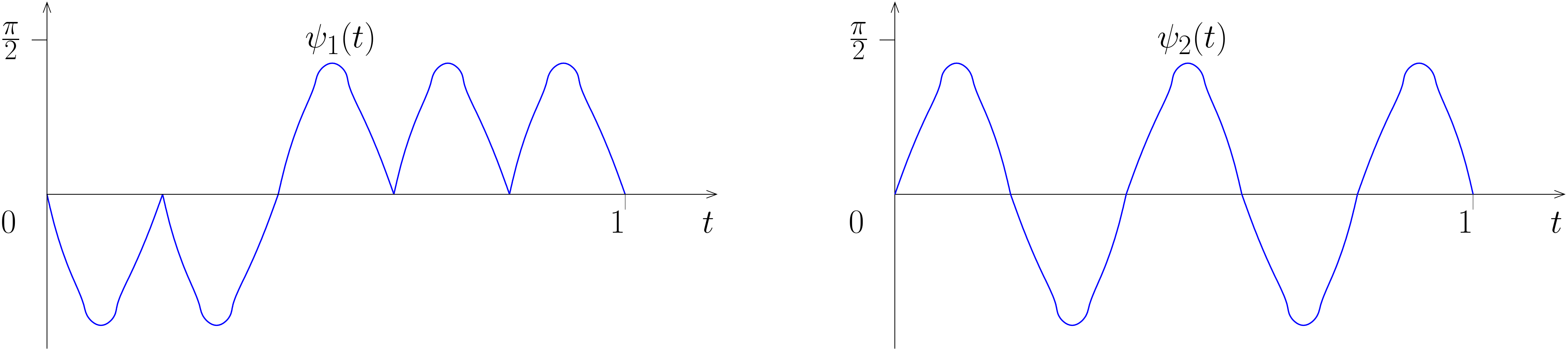}
\caption{\label{fig:equiv} Example of non-uniqueness for nematic crystals. If the azimuthal angle $\ttwo$ is constantly zero,
then $T_1(\psi_1)=T_1(\psi_2)$ for any incident angle $\theta$ .}
\end{figure}

\subsection{Scaling assumptions}\label{sec:scales}

In this section, we discuss the scales of the various quantities involved. 
The numerical values of the physical parameters used in this section are taken from 
Lionheart \& Newton~\cite{LIONHEART-NEWTON-07}.

\textbf{Frequency.} For a He-Ne laser of wavelength $0.633~\mu$m and a slab of thickness
$5~\mu$m, we find $\eta=\eta_{\rm e}\approx 8.4\times10^{-2}$.
In what follows, we will therefore assume that
\begin{equation}\label{eq:def-eta}
\eta \mbox{ is a small dimensionless parameter.}
\end{equation}

\textbf{Dielectric parameters.} In the case of a orthorhombic (diagonal) dielectric tensor, we will assume that
\begin{equation}\label{eq:model-orth}
\tilde{\varepsilon}_{22}  = 1 + \eta^\alpha q , \mbox{ where } \sup_{t\in[0,1]} |q (t)| \leq 1, \mbox{ and }\alpha \mbox{ near } 1.
\end{equation}
A similar assumption for the nematic uniaxial LC cell to give a simple form to \eqref{eq:K1} is
\begin{equation}
\label{eq:model-Uni-F}
\delta= \eta^{\alpha}.
\end{equation}
For problem~\eqref{eq:probleme-tilt-N} to match the orthorhombic  assumption \eqref{eq:model-orth}, we can choose
\begin{equation}\label{eq:model-Uni-T}
\frac{\delta}{\sqrt{1-\delta^2}}\left(1+\frac{\delta}{1+\sqrt{1-\delta^2}}\right)={\tilde\eta}^\alpha.
\end{equation}
For the nematic uniaxial LC cell the values are $|\sqrt{\varepsilon_{\perp}}-\sqrt{\varepsilon_{\parallel}}|\approx 0.15$, $n_0\approx1.52$. The scaling
\eqref{eq:model-Uni-F} yields $\alpha\approx0.94$ whereas \eqref{eq:model-Uni-T} leads to $\alpha\approx0.92$, both values are
indeed close to one.

\textbf{Incident angle.} We write the range of the incident angle as follows:
\begin{equation}\label{eq:lambda-theta}
\lambda \in [\sqrt{\tau \eta^{\alpha}},1], \mbox{ and } \tau>1.
\end{equation}
In the experiment considered, $\theta$ varies between from  $0^\circ$ to about $70^\circ$ :
for larger angles the measurements become unreliable. The extremal value $(\cos(70^\circ))^2$ corresponds to
$\tau= \tau_{\rm e} \approx 1.2$.

\textbf{Measurement error.} We will assume that the measured data is accurate up to errors of order
\begin{equation}\label{eq:error-scale}
\mathcal{O}\left( \frac{\eta^{5+\alpha}}{\lambda^{5}} \right)
\end{equation}
Where for any $x>0$, $|\mathcal{O}(x)|\leq C x$, where $C$ is a constant independent of $x$.
In the experimental case considered, this corresponds to a precision of the order of $10^{-7}$
for a normal incidence, and of the order of $10^{-4}$ for the most slanted incidence. This assumption models the fact that the
measurements become less accurate as slant of the slab increases.

\textbf{Sampling rate.} We suppose that $\lambda$ is measured with a fine sampling rate, e.g. $\eta^2$. Experimentally,
200 incident angles $\theta$ between $-70^\circ$ and $70^\circ$ are used, corresponding to sampling rate $\eta^{1.86}$.

\subsection{Model Problems}
To summarize the discussion of the previous section, we now write down three traceable reconstruction problems pertaining to the orthorhombic
LC cell model and the nematic uniaxial LC cell model.

\begin{quote}
\textbf{Problem~A.} Let $\tau>1$ be a constant, $\alpha$ a parameter close to $1$, and $\eta>0$ a small parameter.
Let $q:[0,1]\to [-1,1]$ be a piecewise continuous function. For every $\lambda\in [\sqrt{\tau \eta^{\alpha}},1]$, let
$T(\lambda)$ and $R(\lambda)$ be the solutions of
\begin{eqnarray}
\eta\lambda \frac{d u_{1}}{d t} & =& \ii \left(\lambda^2 +\eta^\alpha q\right) u_{2} \mbox{ in } [0,1] \nonumber\\
 \eta \frac{d u_{2}}{d t} & =& \ii \,\lambda \, u_{1}  \mbox{ in } [0,1] \label{eq:problemeA}\\
u_{1}(0) & =& 1-R_{2},\quad u_{2}(0)  =1+R_{2},\quad u_{1}(1) =u_{1}(1)=T(\lambda). \nonumber
\end{eqnarray}
What can be determined about $q$  from $\mathcal{D}_A(\lambda)=|T(\lambda)|+ \mathcal{O} \left(\eta^{5+\alpha}\lambda^{-5} \right)$?
\end{quote}
\begin{rem}
We note that the lower bound on $\lambda$, coming from experimental considerations, has a natural analytic interpretation.
When $\lambda^2\gg \eta^\alpha q$, Problem~\eqref{eq:problemeA} corresponds to a wave propagation problem, slightly perturbed by $q$.
We are therefore in an 'optical' regime: the transmitted wave is very similar to the incident wave. When
$\lambda^2$ and $\eta^\alpha q$ are of the same order: the parameter $q$ is no longer a perturbation but a leading order term.
Problem~\eqref{eq:problemeA} becomes diffusive, and one can therefore expect the polarimetric measurements to become unreliable.
\end{rem}

Problem~A is directly inspired by proposition~\ref{pro:nonu-orth} for the orthorhombic LC cell model.
We argued that when the dielectric tensor is diagonal,
with principle axes aligned with the coordinate axes, we cannot hope to reconstruct all three diagonal coefficients.
In particular, the entries $\varepsilon_{11}$ and $\varepsilon_{33}$ cannot be determined independently. According to
proposition~\ref{pro:nonu-orth} the second diagonal entry $\varepsilon_{22}$ determines uniquely $|T_2|$,  independently 
of $\varepsilon_{11}$ and $\varepsilon_{33}$. The scaling assumption were discussed in section~\ref{sec:scales}.

Problem~A is also relevant for the  in-plane nematic uniaxial LC cell model, that is when $\mathbf{n}=(\cos(\tone),0,\sin(\tone))$. In
that case $q=\cos(2\psi)$.
More precisely, proposition~\ref{pro:nonu-nema} shows that $|T(\lambda)|=|T(\lambda)F(\lambda)|$ is measurable, and determined by
Problem~A using the scaling assumption discussed in section~\ref{sec:scales} with $\tilde\eta$ instead of $\eta$ and
$$
q= \left(\cos(2\psi(\mu)) + \frac{\delta}{1+\sqrt{1-\delta^2}}\right)\left(1+\frac{\delta}{1+\sqrt{1-\delta^2}}\right)^{-1},
$$
with $\delta$ and $\tilde\eta$ related by \eqref{eq:model-Uni-T}, and $\mu$ given by proposition~\ref{pro:nonu-nema}. At first order,
$\mu$ is the identity and $\delta\approx \eta^\alpha$. Problem~\eqref{eq:model-Uni-T} is therefore a close variant of
Problem~\eqref{eq:problemeA}.

We already know from of the various non-uniqueness examples presented in section~\ref{sec:orthlc} and \ref{sec:uniLC} that
little can be determined about the interior values of $q$ from Problem~A. We discuss it further, assuming $q$ is smooth in
section~\ref{sec:ProblemA}.
We will see in section~\ref{sec:ProblemB} that we can give a more precise answer to the related problem
\begin{quote}
\textbf{Problem B.} Let $\tau>1,\alpha>0$ be constants, $\eta>0$ a small dimensionless parameter, and $q$ a piecewise continuous function such that $|q|<1$.
For every $\lambda\in [\sqrt{\tau \eta^{\alpha}},1]$, let  $\mathcal{D}_B[q]$ be given by
\begin{eqnarray}
  \mathcal{D}_B[q]: &[\sqrt{\tau \eta^{\alpha}},1] &\to \mathbb{R} \nonumber\\
  &\lambda &\to\int_0^1 \sqrt{1+\frac{\eta^{\alpha}}{\lambda^2} q(s)} ds.\label{eq:defDB}
\end{eqnarray}
What are sufficient conditions so that $\mathcal{D}_B[q_1]=\mathcal{D}_B[q_2]$ implies $q_1=q_2$?
\end{quote}

Finally, let us consider the phase information available from the nematic uniaxial LC cell data.
\begin{quote}
\textbf{Problem C.} Let  $\mathcal{D}_C[\psi]$ be given by 
\begin{eqnarray}
  \mathcal{D}_C[\psi] = \int_0^1 \frac{\sin \psi(s)}{1+\delta \cos \psi(s)} ds +\mathcal{O}(\delta^2),\label{eq:defDC}
\end{eqnarray}
where $\delta$ is a small parameter. Assuming that $\psi:[0,1]\to[-\pi,\pi]$ changes sign once, at $s^*\in(0,1)$
and $|\psi|$ is known, find possible values for $s^*$.
\end{quote}
As both $T(\lambda)F(\lambda)$ and $T(\lambda)/F(\lambda)$ are measured, $F^2(\lambda)$ is available.
This data is measured precisely for close to normal incident angles, that is, $\eta^{-1}\delta\sqrt{1-\lambda^2}$ small,
we can extract from \eqref{eq:K1} the constant depending on $\psi$, which is $\mathcal{D}_C[2\psi]$.
This problem is naturally very simple to solve: we highlight it here to point out how this information can be extracted from the
data. On the other hand, not much more can be obtained from this problem, since the data in this case is just one value.

Define
$$
G(t)=\int_0^t \frac{\sin|\psi(s)|}{1+\delta\cos|\psi(s)|} ds.
$$
Since $0\leq|\psi(s)|\leq \pi$, G is non decreasing, and strictly increasing if $\psi$ does not equal $0$ or $\pi$ on a set of positive measure.
There are two possibilities, depending on whether $\psi$ is positive and then negative, or negative and then positive.
$$
G(s_0) = \frac{1}{2}G(1)+ \frac{1}{2}(\mathcal{D}_C[\psi] + \mathcal{O}(\delta^2))
\mbox{ or }
G(s_1)=  \frac{1}{2} G(1) -  \frac{1}{2}\mathcal{D}_C[\psi]+\mathcal{O}(\delta^2).
$$
If $\delta$ is small enough so that $-G(1)<\mathcal{D}_C[\psi] + \mathcal{O}(\delta^2) <  G(1)$  both $s_1$ and $s_0$ always exists,
are unique if G is strictly increasing.
\section{\label{sec:ProblemA}On Problem~A under a smoothness assumption.}

mIn this section, we show that if $q$ smooth, namely $C^4([0,1])$,
moderate angles of incidence provide information about the
endpoint values of $q$ only. The internal configuration of $q$
is not decidable as it is stable under suitable re-arrangements.
We then explain how Problem~A leads to Problem~B.

Our strategy is to find an   explicit approximate formula for
\begin{equation}\label{eq:defD2}
\mathcal{D}_A^\prime = \frac{4}{|T|^2} -2.
\end{equation}

\begin{notation}
In this section, given $x>0$ and $y>0$, $\mathcal{O}(x^y)$ means $|\mathcal{O}(x^y) x^{-y}|\leq K$, where $K$ is a constant
depending on $\|q\|_{C^{4}([0,1])}$, $\tau$ given by \eqref{eq:lambda-theta} and  $y$ only.
\end{notation}
\begin{prop}\label{pro:approx-wkb-A}
Assume that $q\in C^4([0,1])$. There exist $\eta_0>0$ such that for all $0<\eta\leq\eta_0$ and $0<\alpha<2$,
Problem~A  given by \eqref{eq:problemeA} is equivalent to the reconstruction of $q$ from
\begin{equation}\label{eq:D2}
 \mathcal{D}_A^{\prime} = \left(\frac{\eta}{\lambda}\frac{ d v_1}{d t} (1)\right)^2 + \left(\frac{\eta}{\lambda}\frac{ d v_2}{d t} (1)\right)^2 +
 \left(v_1(1)\right)^2 + \left(v_2(1)\right)^2
+ \mathcal{O}\left( \frac{\eta^{5+\alpha}}{\lambda^{5}} \right),
\end{equation}
where $v_1$ and $v_2$ are given by
\begin{eqnarray*}
 v_1(t) &=&\sqrt{\frac{C(\lambda,0)}{C(\lambda,t)}}\cos\left(\frac{\lambda}{\eta} \int_0^t C(\lambda,s) ds\right)+ \frac{\eta}{2\lambda}\frac{dC}{dt}(\lambda,0)\frac{1}{C(\lambda,0)}v_2(t),\\
 v_2(t) &=& \sqrt{\frac{1}{C(\lambda,0)C(\lambda,t)}}\sin \left(\frac{\lambda}{\eta} \int_0^t C (\lambda,s) ds\right),
\end{eqnarray*}
with
\begin{equation}\label{eq:def-ABC}
 A= \left(1+ \frac{\eta^\alpha}{\lambda^2}q\right)^{-1/4}, \quad B =- \frac{1}{4} A^{3} \frac{d^2 A}{d t^2},\mbox{ and } C = (A (1+\frac{\eta^2}{\lambda^2} B))^{-2}.
\end{equation}
We have
$$
\eta_0 > \max \{\eta  \,:\, 1+\eta^{2-\alpha} B>0 \mbox{ for all } t\in[0,1]\}.
$$
\end{prop}
\begin{rem}
This is an approximation of WKB-type. The uniform pointwise error estimate involves
second order derivatives of $v_1$ and $v_2$ and requires $q\in C^4([0,1])$.
\end{rem}
We will prove this proposition in section~\ref{sec:proofss}. Note that because up to error terms,
$\mathcal{D}_A^\prime$ depends only on  $q$ via the value of $(C,\frac{d C}{d t})$ at $t=0$ and $t=1$ and
on $\int_0^1 C(\lambda,s)ds$, only a large class of equivalent $q$
can be determined in Problem~A,  at best.
\begin{cor}
The usable data contains only some information of the first three derivatives of $q$ at $t=0$ and $t=1$, and on
\begin{equation}\label{eq:reangt1}
\int_0^1 F\left(\frac{\eta^\alpha}{\lambda^2}q, \frac{\eta^\alpha}{\lambda^2}\frac{dq}{dt}, \frac{\eta^\alpha}{\lambda^2}\frac{d^2q}{dt^2}\right) dt,
\end{equation}
where $F$ is an explicit algebraic function, independent of $\lambda$. Since \eqref{eq:reangt1} is stable
under sufficiently smooth rearrangements of $q$ (see e.g. Example~\ref{exa:bubble}) the determination of $q$ from
the data is not possible without additional a priori information on the variations of $q$.
\end{cor}
\begin{exa}\label{exa:bubble}
Let $q_0\in C^4([0,1])$ be a positive function  which is constant on the interval $[t_0,t_1]$ with $1<t_0<t_1<0$. Let $q_1\in C^4(\mathbb{R})$ be positive function with compact support in
$[0,(t_1-t_0)/2]$. Consider the family of functions $q_s$ given by
$$
q_s(t)=q_0(t) \mbox{ for } t\in[0,t_0]\cup[t_1,1], \mbox{ and } q_s(t) = q_0(t) + q_1(t-s)\mbox{ for } t\in(t_0,t_1)
$$
where $s$ is a parameter in $(t_0, (t_0+t_1)/2)$. Then for all such $s$, $q_s\in C^4([0,1])$ has the same end-point values, and for any smoooth functional $F$,
$$
\frac{d}{ds} \int_0^1 F\left(\frac{\eta^\alpha}{\lambda^2}q_s, \frac{\eta^\alpha}{\lambda^2}\frac{dq_s}{dt}, \frac{\eta^\alpha}{\lambda^2}\frac{d^2q_s}{dt^2}\right) dt=0.
$$
Figure~\ref{fig:reagt} represents such a construction for two different values of $s$. Note that the order of accuracy is not the issue when $q$
is smooth, as the WKB-type ansatz can be continued to obtain any order of accuracy, for a suitable $\eta_0$. Counter examples to uniqueness such as the one depicted
in Figure~\ref{fig:reagt} would still hold.
\end{exa}

\begin{figure}\label{fig:reagt}
\includegraphics[width=0.9\columnwidth]{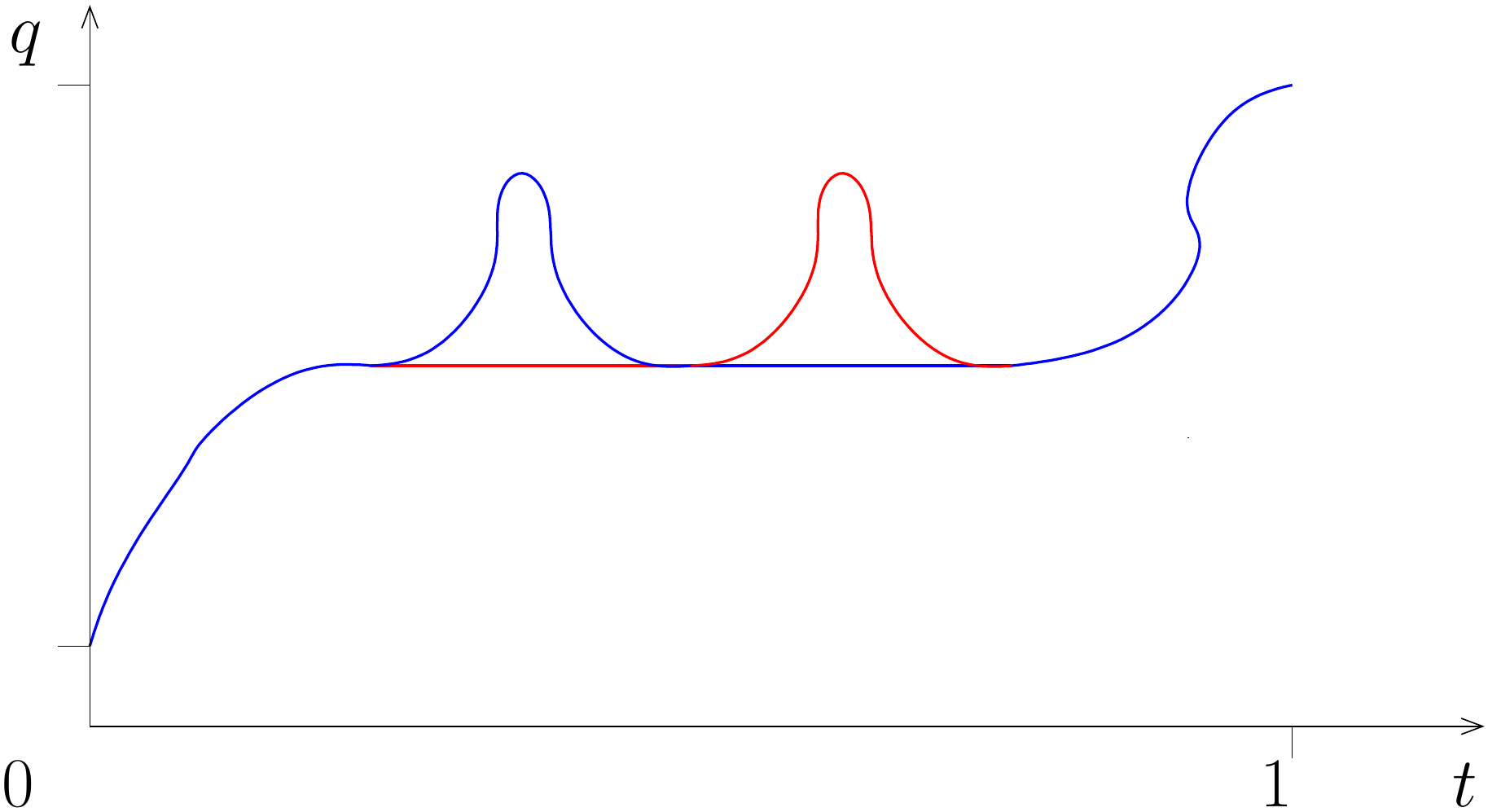}
\caption{Example of non-unique determination from the data. If the dielectric permittivity $q$ follows the blue curve, or the red variant, the available data
$\mathcal{D}_A^\prime$ is the same, for any $\lambda>0$.}
\end{figure}

The main result in this section is devoted to the extraction of a few properties of $q$
at $t=0$ and $t=1$ from moderate angles of incidence, for a sufficiently smooth $q$.

\begin{prop}\label{pro:moderate}
Suppose that $q\in C^4([0,1])$, and that $\alpha$ is close to one. Using the data available for $\frac{1}{2}\leq \lambda \leq 1$,
in Problem~A given by \eqref{eq:problemeA}, the parameters $A_i, i=1,\ldots,5$ given by
\begin{eqnarray*}
A_1 &=& q(0)+q(1) + \mathcal{O}(\eta^{2}),\\
A_2 &=& q(0)q(1) + \mathcal{O}(\eta^{2}), \\
A_3 &=& \frac{d \ln q}{dt}(1)-\frac{d \ln q}{dt}(0) + \mathcal{O}(\eta),\\
A_4 &=& \int_0^1 q(s)ds + \mathcal{O}(\eta^{3-\alpha}),\\
A_5 &=& \int_0^1 q^2(s)ds + \mathcal{O}(\eta^{4-2\alpha}).
\end{eqnarray*}
can be determined.
\end{prop}
The proof of this proposition is given in section~\ref{sec:proofss}.
\begin{rem}
This first order data respects the invariance by mirror symmetry discussed in Proposition~\ref{pro:nonu-orth-2}. If the LC cell is flipped, its index
becomes $\tilde{q}(t)=q(1-t)$, and the parameters $A_i, i=1,\ldots,5$ are unchanged. The fact that only a handful of moments can be recovered from
moderate incidence is known for related more general problems, see Sharafutdinov~\cite{SHARAFUTDINOV-96}.
\end{rem}

To investigate the problem further, let us assume that the map $\lambda\to C(\lambda,t)$ defined in \eqref{eq:def-ABC} is known at the endpoints $t=0$ and $t=1$
 up to the order of approximation. The usable data $\mathcal{D}_A^\prime$ takes the form
$$
\mathcal{D}_A^\prime = K_1 \cos\left(\frac{2\lambda}{\eta} \int_0^1 C (\lambda,s) ds +K_2 \right) +  K_3 + \mathcal{O}\left( \frac{\eta^{5+\alpha}}{\lambda^{5}} \right),
$$
where $K_1$, $K_2$ and $K_3$ are known parameters. Since the data is sampled in $\lambda$ at a fine $\eta^2$ rate, the leading order $\lambda$ dependent data is
\begin{equation}\label{eq:preC}
\int_0^1 C (\lambda,s) ds
\end{equation}
which, at leading order, is $\mathcal{D}_B[q](\lambda)$ given by \eqref{eq:defDB}. Note that this problem also arises naturally if $T$, and not simply its amplitude, 
was considered in Problem~A, as \eqref{eq:preC} is precisely the rate of change of the phase of $T$. 

\section{\label{sec:ProblemB} On Problem B : uniqueness under a monotonicity assumption}
We now turn to the reconstruction of $q$ from $\mathcal{D}_B[q]$
.
\begin{thm}\label{thm:levelsetuniq} Consider $\lambda\to\mathcal{D}_B[q](\lambda)$ for all $\lambda\in [\sqrt{\tau \eta^{\alpha}},1]$ as defined in Problem~B given by \eqref{eq:defDB}.
\begin{itemize}
\item The map $q\to \mathcal{D}_B[q]$ is invariant  under level-set preserving rearrangements of $q$.  That is, given two functions $q_1$ and $q_2$ defined on $[0,1]$
such that $\sup |q_i|\leq1$ which satisfy
$$
\textrm{meas}\left\{ q_1(t)=y,t\in [0,1]\right\} = \textrm{meas}\left\{q_2 (t)=y,t\in [0,1]\right\}\quad \forall y\in [-1,1],
$$
we have $ \mathcal{D}_B[q_1]\equiv \mathcal{D}_B[q_2]$.

\item If  $q\in C^1([0,1])$ is strictly increasing,  $\sup |q|\leq1$,  and $q^{-1} \in C^1([q(0),q(1)])$, then $\mathcal{D}_B[q]$ determines $q$ uniquely.
The reconstruction problem can be formulated in terms of classical transforms in this case. There holds 
\begin{equation}\label{eq:Hankel-Laplace}
 \frac{1}{\lambda}\mathcal{D}_B[q](\lambda) + \frac{d \mathcal{D}_B[q]}{d\lambda} (\lambda) = \frac{2}{\sqrt{\eta^{\alpha}}} \mathcal{L} F_0 Q
\left(\frac{\lambda}{\sqrt{\eta^{\alpha}}}\right),
\end{equation}
where $\mathcal{L}$ is the Laplace transform, $\mathcal{L}f (s)=\int_0^\infty e^{-st} f(t)dt$,
and $F_0$ is the first Hankel transform, $F_0f (s) =\int_0^\infty t f(t) J_0(st) dt$,
where $J_0$ is the Bessel function of the first kind of order $0$, and $Q$ is given by
$$
Q:= x \to \mathbf{1}_{(q(0),q(1))}(x^2)  \frac{ d q^{-1}}{d x}(x^2).
$$
\end{itemize}
\end{thm}
\begin{rem}
The Hankel Transform is self-invertible, and stable. Inverting the Laplace transform is well known to be an exponentially ill-posed task 
see e.g. Epstein \& Schotland \cite{EPSTEIN-SCHOTLAND-08}.
Even more so when, as it is the case here, the Laplace data is limited to an interval $(\sqrt{\tau},\eta^{-\alpha/2})$ located away from zero.
The data $\mathcal{D}_B[q]$ can be written as a linear integral operator
acting on $\frac{d q^{-1}}{dy}$ in the form
$$
\mathcal{D}_B (\lambda) = K\frac{d q^{-1}}{dy} (\frac{\eta^\alpha}{\lambda^2}) \quad \mbox{ with } Kf(x)=\int_{q(0)}^{q(1)} \sqrt{1 +xy}  f(y) dy \mbox{ for } 0\leq x\leq1.
$$
It is clear that $K:C^0([q(0),q(1)])\to C^0([\eta^\alpha,\tau^{-1}])$  is a compact operator : its inverse is therefore unbounded. Including grazing angles
measures, corresponding to $\tau=1$ would still lead to an unstable problem.
\end{rem}
To prove this result, we note that the available data takes a much simpler form after an integral transformation.
\begin{lem}\label{lem:Ktransform}
Given a function $q$ such that defined on $[0,1]$ such that $\sup_{x\in[0,1]}|q (t)|\leq1$, we can define $\hat{\mathcal{D}}_B[q]$ as the
complex valued holomorphic extension of $\mathcal{D}_B[q]$ on the annulus $\sqrt{\eta^{\alpha}\tau}<|\zeta|<1$.

Let $\hat{\mathcal{K}}$ represent the holomorphic extension of
$$
\mathcal{K}: x\to 1 + 2\sqrt{x} e^{x} \int_0^{\sqrt{x}} e^{-t^2} dt
$$
to the same annulus.

Let $\mathcal{D}_B^\prime [q]\in C^\infty(\mathbb{R})$ be given by
$$
\mathcal{D}_B^\prime [q](t) = \int_0^1 e^{\displaystyle - \ii t q(s) } ds \quad  \mbox{for all } t\in\mathbb{R}.
$$
Then, for any $t\in\mathbb{R}$ and $\rho \in (\sqrt{\tau \eta^{\alpha}}, 1)$, we have
$$
\mathcal{D}_B^\prime[q](t) = \frac{1}{2\pi} \int_{0}^{2\pi}  \hat{\mathcal{D}}_B [q]\left(\rho e^{\ii s}\right)
\hat{K}\left( \ii \eta^{-\alpha} \rho^2 t e^{ \ii 2 s}\right) ds.
$$
The correspondence between  $\mathcal{D}_B[q]$ and $\mathcal{D}_B^\prime[q]$ is one to one.
\end{lem}
We prove this result in section~\ref{sec:proofss}. This in turn shows that $\mathcal{D}_B[q]$ determines $q$ uniquely within the class of
smooth monotone functions, albeit in a very unstable manner.
\begin{proof}[Proof of theorem~\ref{thm:levelsetuniq}]
The first statement is a straightforward consequence of the coarea formula applied to the map $t\to\mathcal{D}_B^\prime(t)$.
The second statement follows from the change of variable formula as follows
$$
\mathcal{D}_B^\prime[q](t) =\int_0^{1} e^{-\ii t q(s)} ds =  \int_\mathbb{R} e^{-\ii t y} \mathbf{1}_{(q(0),q(1))}(y)  \frac{ d q^{-1}}{d y}(y) dy,
$$
Thus $\mathcal{D}_B^\prime[q] $ is the Fourier transform of $\frac{ d q^{-1}}{d y}$ on the range of $q$. Applying the inverse Fourier transform we therefore recover
the values of  $q(0)$, $q(1)$ and $\frac{ d q^{-1}}{d y}$, and therefore $q$. We do not claim that the requirement $q^{-1} \in C^{1}([q(0),q(1)])$ is sharp.

Let us now turn to the equation satisfied by $\mathcal{D}_B[q]$. We write,
$$
\frac{1}{\lambda}\mathcal{D}_B[q](\lambda) + \frac{d \mathcal{D}_B[q]}{d\lambda} (\lambda) =
\int_{0}^{1} \frac{1}{\sqrt{\eta^{\alpha}}}\frac{1}{\sqrt{\left(\frac{\lambda}{\sqrt{\eta^{\alpha}}}\right)^{2} + q(y)}} dy.
$$
Using the same change of variable as above, it follows that
\begin{eqnarray*}
\frac{1}{\lambda}\mathcal{D}_B[q](\lambda) + \frac{d \mathcal{D}_B[q]}{d\lambda} (\lambda) &=&
 \frac{2}{\sqrt{\eta^{\alpha}}}\int_0^\infty \frac{1}{\sqrt{\left(\frac{\lambda}{\sqrt{\eta^{\alpha}}}\right)^{2} + v^2}} f(v) v dv \\
 &=&
 \frac{2}{\sqrt{\eta^{\alpha}}}\int_0^\infty e^{-\frac{\lambda}{\sqrt{\eta^{\alpha}}} r}F_0f(r)  dr \\
 &=&
 \frac{2}{\sqrt{\eta^{\alpha}}}\mathcal{L} F_0 f \left( \frac{\lambda}{\sqrt{\eta^{\alpha}}}\right),
\end{eqnarray*}
as announced.
\end{proof}

\section{Proofs for Proposition~\ref{pro:approx-wkb-A}, Proposition~\ref{pro:moderate} and Lemma~\ref{lem:Ktransform}}\label{sec:proofss}
We start by expressing the available data in terms of the fundamental solutions of the ordinary differential system \eqref{eq:problemeA}.
Let $u_1$ and $u_2$ be the two solution of
\begin{eqnarray}
\eta^2  \frac{d^2 u_i}{d t^2} + (\lambda^2 +\eta^\alpha q) u_i=0,  &\,& i=1,2, \label{eq:uv}\\
u_1(0)  = 1, \quad \frac{ d u_1}{d t} (0) =0,  \quad u_2(0)  = 0, \quad \eta \frac{ d u_2}{d t} (0) &=& \lambda. \nonumber
\end{eqnarray}
We have the following identity
\begin{prop}\label{pro:formulaT2}. The transmission parameter $T_2$ determined by \eqref{eq:problemeA} satisfies
$$
\frac{2}{T_2} = \ii\frac{\eta}{\lambda}\frac{ d u_1}{d t} (1) + \frac{\eta}{\lambda}\frac{ d u_2}{d t} (1) +   u_1(1) - \ii u_2(1)  .
$$
where $u_1$ and $u_2$ are determined by \eqref{eq:uv}. Furthermore,
$$
\frac{4}{|T_2|^2} = \left(\frac{\eta}{\lambda}\frac{ d u_1}{d t} (1)\right)^2 + \left(\frac{\eta}{\lambda}\frac{ d u_2}{d t} (1)\right)^2 +   \left(u_1(1)\right)^2
+ \left(u_2(1)\right)^2 +2.
$$
\end{prop}
This is a simple calculation performed in \ref{Annex:B}.

\begin{proof}[Proof of Proposition~\ref{pro:approx-wkb-A}]
We introduce a perturbed problem closely related to the original problem when
$q$ is smooth.
We note $C$ given by \eqref{eq:def-ABC} is well defined and bounded provided $\eta$ is small enough, namely when $1+\frac{\eta^2}{\lambda^2} B$
is bounded below by a positive constant. Assume that $\eta\leq\eta_0$, with $\eta_0$ chosen so that $C$ is bounded.
A tedious but straightforward computation shows that $v_1$ and $v_2$ are solutions of
\begin{eqnarray}
\eta^2  \frac{d^2 v_i}{d t^2} + (\lambda^2 +\eta^\alpha \tilde{q}) v_i=0, &\,& i=1,2,\label{eq:uv-tilde}\\
v_1(0)  = 1, \quad \frac{ d v_1}{d t} (0) =0,  \quad v_2(0)  = 0, \quad \eta \frac{ d v_2}{d t} (0) &=& \lambda \nonumber,
\end{eqnarray}
with
\begin{equation}\label{eq:error-qqt}
\|\tilde{q}-q\|_{C^{0}([0,1])} \leq K\frac{\eta^{4}}{\lambda^{4}},
\end{equation}
where $K$ depends on  $\|q\|_{C^{4}([0,1])}$.
Introducing the function $G(x,y)\in C^{4}([0,1]\times[0,1])$ given by
\begin{eqnarray*}
G(x,y) &=&  \left(v_2(x) v_1(y) - v_1(x) v_2(y)\right)\left(\tilde{q}(y) - q(y) \right), \\
       &=& \sqrt{\frac{1}{C(\lambda,x)C(\lambda,y)}}\sin \left(\frac{\lambda}{\eta} \int_y^x C(\lambda,s) ds \right)\left(\tilde{q}(y) - q(y) \right),
\end{eqnarray*}
we find, comparing \eqref{eq:uv} and \eqref{eq:uv-tilde} and using Duhamel's Formula, that for $i=1,2$ and $t\in[0,1]$ there holds
$$
u_i(t) - v_i(t) - \frac{\eta^{1+\alpha}}{\lambda}\int_0^{1} G(t,y)  \left(u_i(y) -v_i(y) \right) dy=
\frac{\eta^{1+\alpha}}{\lambda}\int_0^{1} G(t,y) v_i(y)  dy.
$$
And from the bound \eqref{eq:error-qqt} we deduce that
$$
\frac{\eta}{\lambda}\left\|\frac{du_i }{d t} -\frac{ d v_i}{d t}\right\|_{C^{0}([0,1])} +\|u_i -v_i\|_{C^{0}([0,1])} \leq K  \frac{\eta^{5+\alpha}}{\lambda^{5}}.
$$
Proposition~\ref{pro:formulaT2} then shows that
$$
\frac{2}{T} = \ii\frac{\eta}{\lambda}\frac{ d v_1}{d t} (1) + \frac{\eta}{\lambda}\frac{ d v_2}{d t} (1) +   v_1(1) - \ii v_2(1)
+ \mathcal{O}\left( \frac{\eta^{5+\alpha}}{\lambda^{5}} \right),
$$
and since $T\approx 1$, this approximation also gives an approximation for $T$ of the same order, which is the precision up to which $T$ is known.
This in turn leads to the formula for $\mathcal{D}_A^\prime$  announced in Proposition~\ref{pro:approx-wkb-A}.
\end{proof}

To prove proposition~\ref{pro:moderate}, we focus on  what can be recovered from $\mathcal{D}_A^\prime$ around the normal incidence.
\begin{prop}\label{pro:modD2}
Suppose that $\frac{1}{2}\leq \lambda \leq 1$, and $\alpha$ is close to one. Then, the parameters $A_i, i=1,\ldots,5$ given by
\begin{eqnarray*}
A_1 &=& q(0)+q(1) + \mathcal{O}(\eta^{2}),\\
A_2 &=& q(0)q(1) + \mathcal{O}(\eta^{2}), \\
A_3 &=& \frac{d \ln q}{dt}(1)-\frac{d \ln q}{dt}(0) + \mathcal{O}(\eta),\\
A_4 &=& \int_0^1 q(s)ds + \mathcal{O}(\eta^{3-\alpha}),\\
A_5 &=& \int_0^1 q^2(s)ds + \mathcal{O}(\eta^{4-2\alpha}).
\end{eqnarray*}
can be extracted from the data $\mathcal{D}_A^\prime$ in this moderate range of incidences.
\end{prop}

\begin{proof}
A Taylor expansion of $\mathcal{D}_A^\prime$ given \eqref{eq:D2} around $\eta=0$, using the fact that $\alpha\geq1$ shows that
\begin{eqnarray*}
\mathcal{D}_A^\prime(\lambda) &=& 2+  \frac{\eta^{2\alpha}}{4\lambda^4}\left(q(1)^2+q(0)^2 -\left(q(1)^3+q(0)^3 \right)\frac{\eta^{\alpha}}{\lambda^2}\right) \\
&+& q(0)q(1)\frac{\eta^{2\alpha}}{4\lambda^4} \left( (q(0)+q(1))\frac{\eta^{\alpha}}{\lambda^2} -2\right)\cos(2\frac{\lambda}{\eta} +\phi(\lambda)) \\
&+& \left(\frac{d q}{dt}(1) q(0) - \frac{dq}{dt}(0)q(1) \right)\frac{\eta^{2\alpha+1}}{\lambda^5}\sin(2\frac{\lambda}{\eta} +\phi(\lambda)) + \mathcal{O}(\eta^{2+2\alpha}),
\end{eqnarray*}
where
$$
\phi(\lambda) = 2\frac{\lambda}{\eta}(\int_0^{1} C(\lambda,s)ds -1 ).
$$
Using the formula for $C(\lambda,s)$ given by \eqref{eq:def-ABC} we find that for $\lambda\in[1-\pi N\eta,1]$, for a given $N$, we have
$$
\phi(\lambda) - \phi(1) = \frac{1-\lambda}{\lambda}\eta^{\alpha-1}\int_0^{1}q(s)ds + \frac{\lambda^3-1}{4\lambda^3} \eta^{2\alpha-2}\int_0^{1}q^2(s)ds +\mathcal{O}(\eta^2).
$$
Without changing the order of approximation, since $\eta$ and $\alpha$ are known parameters, the available data thus becomes
\begin{eqnarray}
\tilde{\mathcal{D}}_2(\lambda) &=& \frac{4\lambda^4}{\eta^{2\alpha}} \left({\mathcal{D}}_2(\lambda) -2\right) \nonumber\\
&=& \left(A_1^2 -2A_2 -A_1\left(A_1^2 -3 A_2 \right)\frac{\eta^{\alpha}}{\lambda^2}\right) \nonumber\\
&+& A_2 \left( A_1 \frac{\eta^{\alpha}}{\lambda^2} -2\right)
\cos\left(2\frac{\lambda}{\eta} + \frac{1-\lambda}{\lambda}\eta^{\alpha-1} A_4 + \frac{\lambda^3-1}{4\lambda^3} \eta^{2\alpha-2}A_5\right) \nonumber \\
&+& A_2 A_3\frac{\eta}{\lambda}\sin\left(2\frac{\lambda}{\eta} + \frac{1-\lambda}{\lambda}\eta^{\alpha-1} A_4 + \frac{\lambda^3-1}{4\lambda^3} \eta^{2\alpha-2}A_5\right)
+ \mathcal{O}(\eta^{2})\nonumber,
\end{eqnarray}
The conclusion follows, as we have $\mathcal{O}(\eta^{-2})$ data points available for $\tilde{\mathcal{D}}_2$.  We could recover these five parameters
 either by a least square fit, or by deriving explicit formulas.
\end{proof}

Finally, we prove lemma~\ref{lem:Ktransform}.
\begin{proof}[Proof of lemma~\ref{lem:Ktransform}]
Since  $|\eta^{\alpha}q(t) /\lambda^2|<\tau<1$ for  every $t \in [0,1]$, we can write using the binomial formula
$$
\mathcal{D}_B[q](\lambda) := \sum_{n=0}^{\infty} (-1)^n \frac{\Gamma\left(n+\frac{1}{2}\right)}{\sqrt{\pi} n !} \frac{\eta^{ \alpha\, n}}{\lambda^{2n}} \int_0^1 (q(s))^n ds
$$
Introducing the holomorphic function $\hat{\mathcal{D}}_3q$ of the complex variable $\zeta$ on the annulus $\sqrt{\eta^{\alpha}\tau}<|\zeta|<1$ given by
$$
\hat{\mathcal{D}}_B [q](\zeta) := \sum_{n=0}^{\infty} (- \eta^{ \alpha})^n \frac{\Gamma\left(n+\frac{1}{2}\right)}{\sqrt{\pi} n !} \zeta^{-2n}  \int_0^1 (q(s))^n ds,
$$
we observe that this quantity is well defined, as the series is absolutely convergent for $|\zeta|>\sqrt{\eta^{\alpha}}$.
As $\lambda$ varies in $[\sqrt{\tau\eta^{\alpha}},1]$, $\mathcal{D}_B [q](\lambda)$ determines $\hat{\mathcal{D}}_B [q]$ on a non-empty real interval
and therefore also on the full annulus. We may therefore define $\hat{\mathcal{D}}_B [q]$ as the
holomorphic extension of $\mathcal{D}_B [q](\lambda)$.

We compute, using the residue formula on the ring of radius $\sqrt{\tau\eta^{\alpha}}<\rho<1$, that for any $t\in\mathbb{R}$,
$$
\frac{(-\ii \eta^{\alpha} t)^n}{n !}\int_{0}^{1} (q(s))^n ds
= \frac{1}{2\sqrt{\pi}} \frac{( \ii \rho^2 t)^n}{\Gamma\left(n+\frac{1}{2}\right)} \int_{0}^{2\pi}  \hat{\mathcal{D}}_B [q]\left(\rho e^{\ii s}\right)e^{2 \ii n s} ds.
$$
This in turn shows, using the fact that all sums are well behaved,  that
\begin{eqnarray*}
\int_0^1 \exp \left( -\ii \eta^{\alpha} t q(s)\right) ds &=& \sum_{n=0}^{\infty} \frac{(-\ii \eta^{\alpha} t)^n}{n !}\int_{0}^{1} (q(s))^n ds \\
&=& \frac{1}{2\sqrt{\pi}} \sum_{n=0}^\infty
\frac{(\ii \rho^2 t)^n}{\Gamma\left(n+\frac{1}{2}\right)} \int_{0}^{2\pi}  \hat{\mathcal{D}}_B [q] \left(\rho e^{\ii s}\right)e^{2 \ii  n s} ds,\\
&=& \frac{1}{2\pi}\int_{0}^{2\pi}  \hat{\mathcal{D}}_B [q] \left(\rho e^{\ii s}\right)  \hat{\mathcal{K}}\left(  \ii \rho^2 t   e^{2 \ii  s}\right) ds,
\end{eqnarray*}
where
$$
\hat{\mathcal{K}}(\zeta) =  \sqrt{\pi} \sum_{n=0}^\infty  \frac{\zeta^n}{\Gamma\left(n+\frac{1}{2}\right)}.
$$
It turns out that $\hat{\mathcal{K}}$ has a closed form expression on the positive real line, in terms of the Error Function. For any $x>0$,
$$
\mathcal{K}(x) = 1 + \sqrt{\pi x} e^{x} \mbox{erf}{\sqrt{x}} = 1 + 2\sqrt{x} e^{x} \int_0^{\sqrt{x}} e^{-t^2} dt.
$$
The converse transformation, using the same steps backwards, gives $\mathcal{D}_B [q]$ in terms of $\mathcal{D}_B^\prime [q]$.  \end{proof}
\section{Conclusion}\label{sec:concl}
We have analyzed the polarimetric measurement with variable incident angle experiment described in 
Lionheart \& Newton~\cite{LIONHEART-NEWTON-07}. After a detailed inspection of the various scales involved in 
the problem, we produced two related one-dimensional reconstruction problems, 
Problem~A given by \eqref{eq:problemeA} and Problem~B given by \eqref{eq:defDB}. We argued that these problems are relevant
for the two Liquid Crystal Cell configurations we considered, orthorhombic and nematic uniaxial. Both problems have a simple 
formulation, and can be studied by applied analysts with no prior exposure to Liquid Crystals.

We partially addressed Problem~A, under a smoothness assumption. Further investigation of this problem would be of practical
and theoretical interest. In the smooth case, it would be interesting to see how the boundary data can be recovered in an optimal manner
numerically. The non-smooth case is left open. While we do not believe that much more can be done with very general parameters, it is very
possible the ill-posedness of this problem is dramatically reduced when additional constraints are imposed, e.g. when the medium
is piecewise linear.

We addressed Problem~B, by characterising a class of equivalent parameters, and by providing a uniqueness result when the coefficient 
is strictly increasing and $C^1$. In that case, we show that the reconstruction amounts to of the inversion of the Laplace Transform 
of the  Hankel Transform of a target function. We have not attempted to address the numerical resolution of such a notoriously 
ill-posed problem: it is very possible that it is manageable when restricted to a particular class of functions.

\section*{Acknowledgements}
The authors were supported by  EPSRC  Science and Innovation award to the Oxford Centre for Nonlinear PDE (EP/E035027/1). 
The second author was also supported by EPSRC Grant EP/E010288/1 and by NSFC grant 11261054.

\appendix
\section{Proof of Proposition~\ref{pro:well-posed-ini}\label{sec:appendixA}}
\begin{prop} \cite{KISELEV-07}
For any two solutions $Z_1$ and $Z_2$ of \eqref{eq:trans-2},  $
\bar{Z}^{T}_1DZ_2$ is independent of $t$, where
$D=\left[\begin{array}{cccc}
{0}&{1}&{0}&{0}\\
{1}&{0}&{0}&{0}\\
{0}&{0}&{0}&{1}\\
{0}&{0}&{1}&{0}
\end{array}\right]$.
\label{annex:prop1}
\end{prop}
\begin{proof}
We compute that
\begin{eqnarray*}
\frac{d}{dt}\left(\bar{Z}_1^TDZ_2\right) &=
\frac{d}{dt}(\bar{Z}_1^T)DZ_2+\bar{Z}_1^TD\frac{d}{dt}(Z_2)\\\
&=-\frac{i}{\epsilon}\bar{Z}_1^TB^TDZ_2+\bar{Z}_1^TD\frac{i}{\epsilon}B Z_2\\
&=\ii\frac{i}{\epsilon}\bar{Z}_1^T(DB-B^TD)Z_2 ,
\end{eqnarray*}
and since $DB=B^TD$, we obtain that $\bar{Z}_1^TDZ_2$ does not depend on $t$.
\end{proof}

We now prove Proposition~\ref{pro:well-posed-ini}.

\begin{proof}[Proof of Proposition~\ref{pro:well-posed-ini}]
Since the entries of the matrix $B$ is are piecewise continuous functions,
global existence and uniqueness of the fundamental solution of the system is well-known see e.g. \cite{ARNOLD-06}.
Now we prove the existence and uniqueness of solutions to the transmission problem via the fundamental solution.
Let us define the fundamental solution of equation \eqref{eq:trans-2} as follows
\[
\Phi(t)=(\phi_1, \phi_2, \phi_3, \phi_4),
\]
where
$\phi_{j}=\left[{\phi_{1j}},{\phi_{2j}},{\phi_{3j}},{\phi_{4j}}\right]^T$
, $j=1,2,3,4$ with $\Phi(0)={\bf{\rm
Id}}.$\\

We have $\Phi(1)\tilde{Z}(0)=\tilde{Z}(1)$ therefore

\begin{equation}
A_1(1)\left[\begin{array}{c}
{R_1}\\{R_2}\\{T_1}\\{T_2}\end{array}\right]=A_2(1)\left[\begin{array}{c}
{0}\\{0}\\{I_1}\\{I_2}\end{array}\right] \label{A1A2}
\end{equation}
where
\[
A_1=\left[\begin{array}{cccc}
{\phi_{12}-\phi_{11}}&{\phi_{14}-\phi_{13}}&{-1}&{0}\\
{\phi_{22}-\phi_{21}}&{\phi_{24}-\phi_{23}}&{-1}&{0}\\
{\phi_{32}-\phi_{31}}&{\phi_{34}-\phi_{33}}&{0}&{-1}\\
{\phi_{42}-\phi_{41}}&{\phi_{44}-\phi_{43}}&{0}&{-1}\end{array}\right]
\]
and \[ A_2=\left[\begin{array}{cccc}
{-1}&{0}&{-(\phi_{11}+\phi_{12})}&{-(\phi_{13}+\phi_{14})}\\
{1}&{0}&{-(\phi_{21}+\phi_{22})}&{-(\phi_{23}+\phi_{24})}\\
{0}&{-1}&{-(\phi_{31}+\phi_{32})}&{-(\phi_{33}+\phi_{34})}\\
{0}&{1}&{-(\phi_{41}+\phi_{42})}&{-(\phi_{43}+\phi_{44})}\end{array}\right].
\]
To establish the announced existence and uniqueness property, let us now show that for any $t\in[0,1]$, $|{\rm det} A_1| \geq 2$.
Let
$U=\left[\begin{array}{c}{U_1}\\{U_2}\\{U_3}\\{U_4}\end{array}\right]$
and
$V=\left[\begin{array}{c}{V_1}\\{V_2}\\{V_3}\\{V_4}\end{array}\right]$
be given by $U_i=\phi_{i2}-\phi_{i1}$ and $V_i=\phi_{i4}-\phi_{i3}$.

Then the determinant of $A_1$ satisfies
\[
{\rm det}(A_1) =\left| \begin{array}{cc} {U_1-U_2}&{V_1-V_2}\\
{U_3-U_4}&{V_3-V_4} \end{array}\right|.
\]
Set
$C_1=\left[\begin{array}{c}  {U_1-U_2}\\{U_3-U_4}\end{array}\right]
$ and $C_2=\left[\begin{array}{c}
{V_1-V_2}\\{V_3-V_4}\end{array}\right]$.
We compute that
\[
   C_1 \cdot
   \bar{C}_2=(U_1\bar{V_1}+U_2\bar{V_2}-U_1\bar{V_2}-U_2\bar{V_1})+(U_3\bar{V_3}+U_4\bar{V_4}-U_3\bar{V_4}-U_4\bar{V_3}).
\]
From the identity $\Phi(0)={\bf {\rm Id}}$ , we deduce that $\bar{V}^T(0)DU(0)=0$,
Since $U$ and $V$ are two solutions of \eqref{eq:trans-2}, Proposition~\ref{annex:prop1} shows that this implies
 $\bar{V}^T D U =0$ for all $t\in[0,1]$.  Expressing this identity in terms of the components of $U$ and $V$, we obtain
\[
 (U_1 \bar{V}_2+U_2 \bar{V}_1)+(U_3 \bar{V}_4+U_4 \bar{V}_3)=0,
\]
and in turn
\begin{equation}
C_1 \cdot
\bar{C}_2= U_1\bar{V_1}+U_2\bar{V_2}+ U_3\bar{V_3}+U_4\bar{V_4} = U \cdot \bar{V}.
\label{c1c1}
\end{equation}
Cauchy-Schwarz inequality thus shows that
 \[ C_1 \cdot \bar{C}_2=\lambda N(U)N(V), \]
with $\lambda \in \mathbb{C}$ and $|\lambda|\leq1$. Similarly, starting from  the identity $\bar{U}^T(0)DU(0)=-2$ and $\bar{V}^T(0)DV(0)=-2$ we obtain $C_1 \cdot \bar{C}_1=N(U)^2+2$ and $C_2
\cdot \bar{C}_2=N(V)^2+2$. Therefore
\begin{eqnarray*}
|{\rm det}(A_1)|^2 &={\rm det}(\bar{A_1}^{T}){\rm det}(A_1)=\left|
\begin{array}{cc}
{C_1 \cdot \bar{C}_1}&{C_2 \cdot \bar{C}_1}\\
{C_1 \cdot \bar{C}_2}&{C_2 \cdot \bar{C}_2}\end{array}\right| \\
&=(1-|\lambda|^2)N(U)^2N(V)^2+4+2N(U)^2+2N(V)^2\\
&\ge 4.
\end{eqnarray*}
\end{proof}
\section{\label{Annex:B}Proof of Proposition~\ref{pro:formulaT2} and Proposition~\ref{pro:nonu-orth-2}}
Problem~\ref{eq:problemeA} can be formulated slightly more generally as follows.
Given an initial condition
\[ u(0)=1+r,v(0)=1-r,\]
find $r\in\mathbb{C}$, $t\in\mathbb{C}$ such that $u(1)=v(1)=t$,
when $u$ and $v$ satisfy\
$$
\frac{d}{d x} u   =  \ii f(x)v,\quad \frac{d}{d x} v    =   \ii g(x)u,
$$
with $f$ and $g$ two positive functions, bounded above and
below by positive constants. This applies to system \eqref{eq:probleme2-0} with $v=z_{3}$, $u=z_{4}$, $\eta f=\lambda$, and $\eta \lambda g = \lambda^2+\varepsilon_{22}-1$.
Note $u$ is  a solution of the real valued elliptic equations
\begin{equation}
\frac{d}{d x}\left(\frac{1}{f}\frac{d}{d x} w\right)+g\, w=0,\label{eq:helm1}
\end{equation}
The general uniqueness results for ordinary differential equations show that
$u$ is a linear combination of the two fundamental solutions
${\rm w}_{1}$ and ${\rm w}_{2}$, of \eqref{eq:helm1} corresponding to the initial
conditions ${\rm w}_{1}(0)=1\, \frac{1}{f(0)}\frac{d}{d x}{\rm w}_{1} (0)=0$ and
${\rm w}_{2}(0)=0,\,\frac{1}{f(0)}\frac{d}{d x} {\rm w}_{2}(0)=1$. \[
u(t)=A{\rm w}_{1}(t)+B{\rm w}_{2}(t).\]
 Note that the
Wronskian of the problem is constant:\[
\frac{1}{f(x)}\left({\rm w}_{1}\frac{d}{d x}{\rm w}_{2}-{\rm w}_{2}\frac{d}{d x} {\rm w}_{1} \right)\equiv 1. \]
Using the initial and final conditions, we obtain
$$
 \left[\begin{array}{c}
A\\
B\end{array}\right] =
 \left[\begin{array}{c}
1+r\\
\ii (1-r)\end{array}\right] \mbox{ and }
\left[\begin{array}{cc}
{\rm w}_{1}(1) & {\rm w}_{2}(1)\\
\frac{1}{f(1)}\frac{d}{d x} {\rm w}_{1}(1) &
\frac{1}{f(1)}\frac{d}{d x} {\rm w}_{2}(1)\end{array}\right] \left[\begin{array}{c}
A\\
B\end{array}\right]=
\left[\begin{array}{c}
t\\
\ii t \end{array}\right]
$$
which leads to the following formulae for $t$ and $r$,
 \[
t=\frac{2}{e_{2}(1)+\ii e_{1}(1)}\mbox{ and } r=\frac{\ii e_{1}(1)-e_{2}(1)}{e_{2}(1)+\ii e_{1}(1)}.
\]
where $e_{i} =\frac{1}{f(x)}\frac{d}{d x} {\rm w}_{i} -\ii {\rm w}_{i}$ for  $i=1,2$. Using the Wronskian identity, we note that $t$
is always well defined, as a simple calculation shows that
\begin{eqnarray*}
\frac{4}{|t|^{2}} & =|e_{1}(1)|^{2}+|e_{2}(1)|^{2}+2\\
 & =\left(\frac{1}{f(1)}\frac{d}{d x} {\rm w}_{2} (1)\right)^{2}
   +\left(\frac{1}{f(1)}\frac{d}{d x} {\rm w}_{1} (1)\right)^{2}
   +\left({\rm w}_{2}(1)\right)^{2}+\left({\rm w}_{1}(1)\right)^{2}
   +2.
\end{eqnarray*}
One can check that $|t|^2+|r|^2=1$, therefore the amplitude of the reflected field, if it was available, would be redundant.

To prove that $t$ is unchanged if $f$ and $g$ are replaced by $\tilde f(x)= f(1-x)$ and $\tilde g(x) = g(1-x)$, note the solution $\tilde u$
corresponding to $\tilde f$ and $\tilde g$ can be written as linear  combination of $w_1(1-x)$ and $w_2(1-x)$.
The associated reflection and transmission coefficients $\tilde r$ and $\tilde t$ satisfy
$$
 \left[\begin{array}{c}
A\\
-B\end{array}\right] =
 \left[\begin{array}{c}
t\\
\ii t
\end{array}\right] \mbox{ and }
\left[\begin{array}{cc}
{\rm w}_{1}(1) & {\rm w}_{2}(1)\\
-\frac{1}{f(1)}\frac{d}{d x} {\rm w}_{1}(1) &
-\frac{1}{f(1)}\frac{d}{d x} {\rm w}_{2}(1)\end{array}\right] \left[\begin{array}{c}
A\\
B\end{array}\right]=
\left[\begin{array}{c}
1+r\\
\ii (1-r) \end{array}\right]
$$
which shows that $t=\tilde t$ and $r=\tilde r$.
\end{document}